\def\draft{1}
\def\conferenceversion{1}

\documentclass[11pt]{article}

\usepackage{etex}
\usepackage{verbatim}
\usepackage{xspace,enumerate}
\usepackage[dvipsnames]{xcolor}
\usepackage[T1]{fontenc}
\usepackage[full]{textcomp}
\usepackage[american]{babel}
\usepackage{mathtools}
\usepackage{amsthm}
\usepackage[normalem]{ulem}
\usepackage{fullpage}
\usepackage{tcolorbox}
\usepackage{todonotes}

\usepackage{mathpazo}
\usepackage{multicol}
\usepackage{comment}
\usepackage{esvect}
\usepackage{xcolor}
\usepackage{cases}
\usepackage{mathrsfs}
\usepackage{ifthen}
\usetikzlibrary{backgrounds}

\usepackage[
letterpaper,
top=1in,
bottom=1in,
left=1in,
right=1in]{geometry}
\usepackage{newpxtext} %
\usepackage{textcomp} %
\usepackage[varg,bigdelims]{newpxmath}
\usepackage[scr=rsfso]{mathalfa}%
\usepackage{bm} %
\let\mathbb\varmathbb
\usepackage{microtype}
\usepackage[pagebackref,bookmarksnumbered,colorlinks=true,urlcolor=blue,linkcolor=blue,citecolor=OliveGreen]{hyperref}
\usepackage[capitalise,nameinlink]{cleveref}
\crefname{lemma}{Lemma}{Lemmas}
\crefname{fact}{Fact}{Facts}
\crefname{theorem}{Theorem}{Theorems}
\crefname{corollary}{Corollary}{Corollaries}
\crefname{claim}{Claim}{Claims}
\crefname{example}{Example}{Examples}
\crefname{algorithm}{Algorithm}{Algorithms}
\crefname{problem}{Problem}{Problems}
\crefname{definition}{Definition}{Definitions}
\crefname{exercise}{Exercise}{Exercises}
\usepackage{amsthm}

\newtheorem{theorem}{Theorem}[section]
\newtheorem*{theorem*}{Theorem}
\newtheorem{lemma}[theorem]{Lemma}
\newtheorem*{lemma*}{Lemma}

\newtheorem*{fact*}{Fact}
\newtheorem{proposition}[theorem]{Proposition}
\newtheorem*{proposition*}{Proposition}
\newtheorem{corollary}[theorem]{Corollary}
\newtheorem*{corollary*}{Corollary}

\newtheorem*{hypothesis*}{Hypothesis}

\newtheorem*{conjecture*}{Conjecture}
\theoremstyle{definition}
\newtheorem{definition}[theorem]{Definition}
\newtheorem*{definition*}{Definition}

\newtheorem*{construction*}{Construction}

\newtheorem*{example*}{Example}

\newtheorem*{question*}{Question}
\newtheorem{algorithm}[theorem]{Algorithm}
\newtheorem*{algorithm*}{Algorithm}

\newtheorem*{assumption*}{Assumption}

\newtheorem*{problem*}{Problem}

\newtheorem*{openquestion*}{Open Question}
\theoremstyle{remark}

\newtheorem*{claim*}{Claim}

\newtheorem*{remark*}{Remark}

\newtheorem*{observation*}{Observation}
\usepackage{paralist}
\let\originalleft\left
\let\originalright\right
\renewcommand{\left}{\mathopen{}\mathclose\bgroup\originalleft}
\renewcommand{\right}{\aftergroup\egroup\originalright}
\usepackage{turnstile}
\usepackage{mdframed}
\usepackage{tikz}
\usetikzlibrary{positioning}
\usetikzlibrary{fit,backgrounds}
\usetikzlibrary{arrows.meta,
                decorations.markings}
\usepackage{tkz-graph}
\usetikzlibrary{arrows}
\usetikzlibrary{fit,shapes.geometric} 

\usetikzlibrary{calc,backgrounds}   
\usepackage{subcaption}
\usepackage{float}
\usepackage{caption}
\DeclareCaptionType{Algorithm}
\usepackage{array}
\usepackage{bbm}
\usepackage{xparse}
\usepackage{amsthm} %
\makeatletter
\let\latexparagraph\paragraph
\RenewDocumentCommand{\paragraph}{som}{%
  \IfBooleanTF{#1}
    {\latexparagraph*{#3}}
    {\IfNoValueTF{#2}
       {\latexparagraph{\maybe@addperiod{#3}}}
       {\latexparagraph[#2]{\maybe@addperiod{#3}}}%
  }%
}
\newcommand{\maybe@addperiod}[1]{%
  #1\@addpunct{.}%
}
\makeatother

\usepackage{boxedminipage}

\newcommand{\Abs}[1]{\left\lvert#1\right\rvert}
\newcommand{\bigabs}[1]{\big\lvert#1\big\rvert}

\newcommand\bdot\bullet

\newcommand{\N}{\mathbb N}
\newcommand{\R}{\mathbb R}

\renewcommand{\leq}{\leqslant}

\renewcommand{\geq}{\geqslant}

\newcommand{\vol}{\text{vol}}
\newcommand{\phid}{\phi_{\mathit{dir}}}
\newcommand{\bdir}{\beta_{\mathit{dir}}}

\let\epsilon=\varepsilon
\numberwithin{equation}{section}
\newcommand\MYcurrentlabel{xxx}
\newcommand{\MYstore}[2]{%
  \global\expandafter \def \csname MYMEMORY #1 \endcsname{#2}%
}
\newcommand{\MYload}[1]{%
  \csname MYMEMORY #1 \endcsname%
}
\newcommand{\MYnewlabel}[1]{%
  \renewcommand\MYcurrentlabel{#1}%
  \MYoldlabel{#1}%
}
\newcommand{\MYdummylabel}[1]{}
\newcommand{\torestate}[1]{%
  \let\MYoldlabel\label%
  \let\label\MYnewlabel%
  #1%
  \MYstore{\MYcurrentlabel}{#1}%
  \let\label\MYoldlabel%
}
\newcommand{\restatedef}[1]{%
  \let\MYoldlabel\label
  \let\label\MYdummylabel
  \begin{definition*}[Restatement of \cref{#1}]
    \MYload{#1}
  \end{definition*}
  \let\label\MYoldlabel
}
\newcommand{\restatetheorem}[1]{%
  \let\MYoldlabel\label
  \let\label\MYdummylabel
  \begin{theorem*}[Restatement of \cref{#1}]
    \MYload{#1}
  \end{theorem*}
  \let\label\MYoldlabel
}
\newcommand{\restatelemma}[1]{%
  \let\MYoldlabel\label
  \let\label\MYdummylabel
  \begin{lemma*}[Restatement of \cref{#1}]
    \MYload{#1}
  \end{lemma*}
  \let\label\MYoldlabel
}
\newcommand{\restateprop}[1]{%
  \let\MYoldlabel\label
  \let\label\MYdummylabel
  \begin{proposition*}[Restatement of \cref{#1}]
    \MYload{#1}
  \end{proposition*}
  \let\label\MYoldlabel
}
\newcommand{\restatefact}[1]{%
  \let\MYoldlabel\label
  \let\label\MYdummylabel
  \begin{fact*}[Restatement of \cref{#1}]
    \MYload{#1}
  \end{fact*}
  \let\label\MYoldlabel
}
\newcommand{\restate}[1]{%
  \let\MYoldlabel\label
  \let\label\MYdummylabel
  \MYload{#1}
  \let\label\MYoldlabel
}

\allowdisplaybreaks
\definecolor{niceish}{HTML}{74b807} 
\ifnum\draft=1
\newcommand{\salil}[1]{\textcolor{blue}{[Salil: #1]}}
\newcommand{\jake}[1]{\textcolor{violet}{[Jake: #1]}}
\newcommand{\tom}[1]{\textcolor{WildStrawberry}{[Tommaso: #1]}}
\newcommand{\chris}[1]{\textcolor{niceish}{[Chris: #1]}}
\newcommand{\jiyu}[1]{\textcolor{Orange}{[Jiyu: #1]}}
\else
\newcommand{\salil}[1]{}
\newcommand{\jake}[1]{}
\newcommand{\tom}[1]{}
\newcommand{\chris}[1]{}
\newcommand{\jiyu}[1]{}
\fi

\newcommand{\om}{\om}

\usepackage[dvipsnames]{xcolor}   %
\usetikzlibrary{intersections}   
\usepackage[dvipsnames]{xcolor}

\RequirePackage[outline]{contour} 
\contourlength{0.065pt}
\contournumber{10}%

\usetikzlibrary{arrows.meta}    

\ifpdf
\hypersetup{
    pdftitle={Directed Spectral Graph Theory},
    pdfauthor={ Jake Ruotolo, Salil Vadhan}
}
\fi

\begin{document}
\date{}

\title{
Singular Values Versus Expansion\\ in Directed and Undirected Graphs
}

\ifnum\conferenceversion=0
\author{Author names omitted}
\else
\author{Jake Ruotolo\thanks{School of Engineering and Applied Sciences, Harvard University. \texttt{jakeruotolo@g.harvard.edu}. Supported by a Simons Investigator grant to S. Vadhan.}
\and Salil Vadhan\thanks{School of Engineering and Applied Sciences, Harvard University.
\texttt{salil\_vadhan@harvard.edu}. Supported by a Simons Investigator grant.}}
\fi

\maketitle

\thispagestyle{empty}
\begin{abstract}
We relate the nontrivial singular values $\sigma_2,\ldots,\sigma_n$ of the normalized adjacency matrix of an Eulerian directed graph to combinatorial measures of graph expansion:
\begin{itemize}
\item We introduce a new directed analogue of conductance $\phid$, and prove a Cheeger-like inequality showing that $\phid$ is bounded away from 0 iff $\sigma_2$ is bounded away from 1.  In undirected graphs, this can be viewed as a unification of the standard Cheeger Inequality and Trevisan's Cheeger Inequality for the smallest eigenvalue.
\item We prove a singular-value analogue of the Higher-Order Cheeger Inequalities, giving a combinatorial characterization of when 
$\sigma_k$ is bounded away from 1.
\item We tighten the relationship between $\sigma_2$ and vertex expansion, proving that if a $d$-regular graph $G$ with the property that all sets $S$ of size at most $n/2$ have at least $(1+\delta)\cdot |S|$ out-neighbors, then $1-\sigma_2=\Omega(\delta^2/d)$.  This bound is tight and saves a factor of $d$ over the previously known relationship.
\end{itemize}
\end{abstract}
\clearpage
\thispagestyle{empty}

\clearpage
\thispagestyle{empty}
\microtypesetup{protrusion=false}
\tableofcontents{}
\microtypesetup{protrusion=true}

\clearpage
\pagestyle{plain}
\setcounter{page}{1}

\section{Introduction}

The notion of graph expansion is pervasive in theoretical computer science.  (See the survey \cite{HLW06}.)  Arguably, one of the reasons for this is its robustness: many different natural formulations of graph expansion are approximately equivalent (e.g. notions of spectral expansion, edge expansion, and vertex expansion).  

Much of the study of graph expansion has focused on undirected graphs.  However, there are settings where expansion of {\em directed graphs} arises naturally, and would benefit from a more well-developed theory.  Examples include the study of nonreversible Markov chains~\cite{mihail1989conductance,fill1991eigenvalue} and the derandomization of space-bounded computation~\cite{ReingoldTrVa06,RozenmanVa05,ChungReVa11,AKMPSV20,HozaPyVa21,PyneVa21-WPRG,chen2023weighted,CL20}. 

\paragraph{Spectral Expansion of Directed Graphs.}  The standard spectral measure of expansion for directed graphs is given by the second-largest singular value of the normalized adjacency matrix, which we will define more precisely now.  Specifically, let $G$ be a (possibly weighted) {\em Eulerian} directed graph, where the (weighted) in-degree of each vertex $v$ equals the (weighted) out-degree, both of which we denote by $d(v)$.  Like with undirected graphs, the stationary distribution of the random walk on an Eulerian graph $G$ is proportional to the degree vector $d$. Moreover, every directed graph can be made Eulerian without changing the random-walk matrix by reweighting the edges according to the stationary distribution.

Let $M$ be the (weighted) adjacency matrix of $G$, and $D$ the diagonal matrix of vertex degrees.  Then the {\em normalized adjacency matrix} of $G$ is $A = D^{-1/2}M D^{-1/2}$. The largest singular value of $A$ is 1, with (left and right) singular vector $d^{1/2}$, the vector whose $v$'th entry is $d(v)^{1/2}$.  We write $\sigma_2(G)$ to denote the second-largest singular value of $A$.  We have $\sigma_2(G)\in [0,1]$ with an ``expander'' being a graph where $\sigma_2(G)$ is bounded away from 1, and smaller values of $\sigma_2(G)$ corresponding to better expansion. Indeed,
it is well-known that random walks on $G$ mix in $\ell_2$ norm (normalized by the stationary distribution) in time $O(1/(1-\sigma_2(G)))$.

The use of singular values (rather than eigenvalues) arises naturally for directed graphs, since when $A$ is assymetric, it need not have an orthonormal basis of eigenvectors.  That said, $\sigma_2(G)$ is also commonly used for undirected graphs, where it is referred to as {\em two-sided} spectral expansion, since it bounds the magnitude of both the positive and negative nontrivial eigenvalues. 

The main question we study in this work is:
\begin{center}
    \textit{Can we establish approximate equivalences between $\sigma_2(G)$ \\and combinatorial measures of expansion?}
\end{center}
A positive answer to this question was known for the regime where $\sigma_2(G)$ is close to 0.  The Expander Mixing Lemma~\cite{ALON-Chung} (See Lemma 4.15 \cite{pseudorandomness_Vadhan}.) shows that the fraction of $G$'s edges that go between any two sets $S$ and $T$ of vertices is approximately equal to the product $\mu(S)\mu(T)$ of the densities of $S$ and $T$, with an error term that vanishes $\sigma_2(G)\sqrt{\mu(S)\cdot \mu(T)}$. Here for $R\subseteq V$, $\mu(R) = \vol(R)/\vol(V)$ denotes the stationary probability mass of $R$, and $\vol(R)=\sum_{v\in R} d(v)$.
Bilu and Linial~\cite{bilu-linial} proved a converse for undirected graphs, showing that if the aforementioned quasirandomness condition holds for all sets $S$ and $T$ with an error term $\alpha\cdot \sqrt{\mu(S)\mu(T)}$, then we have $\sigma_2(G)=O(\alpha\cdot \log(1/\alpha))$. Butler~\cite{BUTLER-LME} generalized the Expander Mixing Lemma and this converse to irregular directed graphs.

Thus, our focus is on the regime where $\sigma_2(G)$ is close to 1, and being an ``expander'' means only that $\sigma_2(G)$ is bounded away from 1. 

\paragraph{A Singular-Value Cheeger Inequality.}  Recall the discrete Cheeger Inequality~\cite{AM85}, which says that for every undirected graph $G$, we have:
$$\frac{1-\mu_2(G)}{2} \leq \phi(G) \leq \sqrt{2(1-\mu_2(G))},$$
where $\mu_2(G)$ is the second-largest {\em eigenvalue} of $G$'s normalized adjacency matrix, and $\phi(G)$ is $G$'s {\em conductance}, defined as follows:
$$\phi(G) = \min_{S : \vol(S)\leq 1/2} \frac{e(S,S^c)}{\vol(S)},$$
where $e(S,T)$ denotes the sum of the edge weights between $S$ and $T$.  Cheeger's Inequality is a robust version of the fact that $\mu_2(G)=1$ iff $G$ is disconnected iff $\phi(G)=0$.

For directed graphs $G$, we introduce a new {\em directed conductance}, denoted $\phid(G)$, which satisfies the following {\em singular-value Cheeger Inequality}:
$$\frac{1-\sigma_2(G)}{2} \leq \phid(G) \leq
\sqrt{2(1-\sigma_2(G))}.$$
Specifically, directed conductance is defined as follows:
$$\phid(G) = \min_{S,T\subseteq V : \mu(S)+\mu(T)\leq 1} \frac{e(S,T^c)+e(S^c,T)}{\vol(S)+\vol(T)}.$$
Notice that $\phid(G)=0$ iff there are sets $S$ and $T$ such that all edges leaving $S$ enter $T$ and all edges entering $T$ come from $S$.  Notice that this implies that there is no volume growth in one step of the random walk: if we start at the stationary distribution conditioned on being in $S$, then after one step we are at the stationary distribution conditioned on being in $T$.  This implies that $\sigma_2(G)=1$, and an extreme case of our singular-value Cheeger Inequality says that it is in fact equivalent to $\sigma_2(G)=1$.

In addition, we prove that if $G$ is undirected, then the minimum in $\phid(G)$ is approximately achieved by sets $S,T$ such that $S=T$ or such that $S$ and $T$ are disjoint.  In the case that $S=T$, $\phid(G)$ becomes equal to ordinary conductance $\phi(G)$.  In case $S$ and $T$ are disjoint, then $\phid(G)$ is within a constant factor of Trevisan's {\em bipartiteness ratio} $\beta(G)$. 
Thus, for undirected graphs, our singular-value Cheeger Inequality can be viewed as a combination of the ordinary Cheeger Inequality and Trevisan's bipartiteness Cheeger Inequality (which relates $\beta(G)$ to $1+\mu_n(G)$, where $\mu_n(G)\leq 0$ is the smallest eigenvalue of $A$, and thus $\sigma_2(G)=\max\{\mu_2(G),-\mu_n(G)\}$).

We also provide singular-value analogues of the {\em higher-order Cheeger Inequalities}~\cite{higher-order-cheeger}. These show that  $\sigma_k(G)$ is close to 1 iff there are two partitions $(S_1,\ldots,S_k)$ and $(T_1,\ldots,T_K)$ of $V$ such that $\max_i \phid(S_i,T_i)$ is small. 

\paragraph{Singular Values versus Vertex Expansion.}  Next we turn our attention to the relation between $\sigma_2(G)$ and {\em vertex expansion}.  A digraph $G$ is a {\em $(k,c)$-vertex expander} if for every set $S$ of vertices with $|S|\leq k$, we have $|N(S)|\geq c\cdot |S|$.  Since vertex expansion is invariant under adding self-loops, we will assume that $G$ is a $d$-regular graph for a constant $d$.
A standard fact (cf. \cite{pseudorandomness_Vadhan}) is that every regular digraph $G$ is an $(n/2,1+\delta)$-vertex expander for $\delta=1-\sigma_2(G)$.  Alon~\cite{AlonEigenvaluesExpanders} proved a partial converse, showing that if a $d$-regular undirected graph is an $(n/2,1+\delta)$-vertex expander, then $1-\mu_2 = \Omega(\delta^2/d)$.
This converse is partial in that it applies only to undirected graphs and only lower bounds $1-\mu_2$ rather than $1-\sigma_2$. But, the bound $1-\mu_2\geq \Omega(\delta^2/d)$ is tight by considering an undirected cycle of length $n = 4/\delta$ with $d-2$ self loops at each vertex.  In \cite{pseudorandomness_Vadhan}, it is observed that we can overcome these limitations by considering $A^TA$, which yields a bound of $1-\sigma_2 = \Omega(\delta^2/d^2)$.  This is worse than Alon's bound by a factor of $d$, coming from the fact that the graph corresponding to $A^TA$ has degree $d^2$ rather than $d$.

In this paper, we show how to save that factor of $d$.  Specifically, we prove that if $G$ is a $(n/2,1+\delta)$ vertex expander, then $1-\sigma_2(G) = \Omega(\delta^2/d)$. Which appears to be a new bound even for undirected graphs.

\paragraph{Proof Techniques.}
Many of our results are obtained by considering the \textit{symmetric lift} $sl(G)$ which is the bipartite graph with $n$ vertices on each side where we connect left vertex $u$ with right vertex $v$ if $(u,v)$ is an edge in $G$. It turns out that $\mu_k(sl(G)) = \sigma_k(G)$, so the symmetric lift gives a simple reduction from bounding singular values of directed graphs to bounding eigenvalues of undirected bipartite graphs, for which we can apply known results.

Alon \cite{AlonEigenvaluesExpanders} used the symmetric lift to relate combinatorial notions of expansion in non-bipartite and bipartite undirected graphs graphs. In an exercise, Tao \cite{tao2015expansion} points out that Lemma~\ref{lem:eig_vals_sl} gives a way to relate ``two-sided" spectral expansion in (directed or undirected) graphs, which is defined in terms of $\sigma_2$, to ``one-sided" spectral expansion defined in terms of $1-\mu_2$. 

In the past decade, the symmetric lift has been a useful tool for studying combinatorial and spectral sparsification of graphs and CSPs. Filtser and Krauthgamer \cite{CSP_krauthgamer17} used it to characterize the sparsifiability of certain binary CSPs via a reduction to cut sparsification of graphs. Their results were extended in \cite{butti2020sparsification, pelleg2023additive} using similar ideas. For spectral sparsification Ahmadinejad et al. \cite{sv_approx} introduced a new notion called singular value (SV) sparsification and used the symmetric lift of a digraph $G$ to reduce directed graph SV-sparsification to undirected graph SV-graph sparsification. The symmetric lift also played a role in \cite{kyng2016sparsified,kyng2022derandomizing}.

\section{Preliminaries}\label{sec:prelim}

\subsection{Graph Notation}
 In this paper, the most general type of graphs we work with are weighted Eulerian directed graphs allowing self-loops. 
 
 \begin{definition}
    A \textbf{weighted directed graph} (\textbf{digraph}) $G = (V,E, w)$ is a directed graph with vertex set $V$, weight function $w:V^2\to\R_{\geq 0}$ and edge set $E = \{(u,v)\in V^2: w(u,v) > 0\}$.  
 \end{definition}
 We call the graph $G$ \textbf{undirected} if the weight function is symmetric, i.e., for all $(u,v)\in V^2$ $w(u,v) = w(v,u)$. We call the graph $G$ \textbf{unweighted} if the weight function only takes values in $\{0,1\}$; in this case we omit the weight function. 
 \par The \textbf{out-degree} of a vertex $v\in V$ is $d^+(v) = \sum_{u\in V}w(v,u)$. The \textbf{in-degree} of $v$ is defined by $d^-(v) = \sum_{u\in V}w(u,v)$. The graph $G$ is called \textbf{Eulerian} if for all $v\in V$, $d^+(v) = d^-(v)$. The \textbf{out-neighborhood} $N^+(v) = \{u\in V\colon (v,u)\in E(G)\}$ and \textbf{in-neighborhood} $N^-(v) = \{u\in V\colon (u, v)\in E(G)\}$. For a subset of vertices $S\subseteq V$, the out-volume is defined by $\vol^+(S) = \sum_{v\in S}d^+(S)$ and the in-volume is $\vol^-(S) = \sum_{v\in S}d^-(S)$. 
For Eulerian graphs we often write $d(v) = d^+(v) = d^-(v)$ and  $\vol(S) = \vol^+(S) = \vol^-(S)$. When the graph is undirected we omit the superscripts.
 
 \par For subsets $A, B \subseteq V(G)$ we define $$E_G(A,B) = \{(a,b)\in A\times B\colon w(a,b) > 0\}$$ and $$e_G(A,B) = \sum_{(a,b)\in A\times B}w(a,b).$$ Observe that when $G$ is undirected, edges with $a,b \in A\cap B$ (other than self-loops) are counted twice. Note that when $G$ is unweighted, we have $e_G(A,B) = \left|E_G(A,B)\right|$. When $A = B$ we write $E_G(A,A)$ and $e_G(A,A)$ as $E_G(A)$ and $e_G(A)$, respectively. 

\subsection{Graph matrices and Spectra}
Let $G = (V,E, w)$ be a weighted Eulerian digraph on $n$ vertices. The \textbf{weighted adjacency matrix} of $G$ is the $n\times n$ matrix $M_G$ defined by $M_G(u,v) = w(u,v)$. The \textbf{normalized adjacency matrix} of $G$ is $A_G = D^{-1/2}_GMD^{-1/2}_G$, where $D_G$ is the diagonal matrix with $D_G(i,i) = d^+(i) = d^-(i)$. The matrix $D_G$ is called the \textbf{degree matrix} of $G$. If $G$ is an undirected graph, then the matrices $M_G$ and $A_G$ are symmetric hence diagonalizable. We denote the eigenvalues and singular values of $A_G$ by $1=\mu_1(G) \geq \mu_2(G)\geq \ldots \geq \mu_n(G)\geq -1$ and $1=\sigma_1(G)\geq \sigma_2(G)\geq \ldots \geq \sigma_n(G)\geq 0$ respectively. We omit the subscript or parenthesized $G$ when the underlying graph is clear. 
\par Earlier in this section we mentioned that we restrict our focus to weighted Eulerian digraphs. In fact, every digraph can be made Eulerian without changing the random walk matrix by a scaling of the edge weights, induced by a vertex re-weighting. Let $G = (V,E,w)$ be a strongly connected digraph and $D_{out}$ be the diagonal matrix such that $D_{out}(u,u) = d^+(u)$. The \textbf{random walk matrix} of $G$ is defined as $W = D^{-1}_{out}M$. By basic Markov chain theory e.g. \cite{levin2017markov}, the random walk on $G$ has a stationary distribution $\pi$ with full support, a positive vector in $\R^V$ such that $\|\pi\|_1 = 1$ and $\pi W = \pi$. Let $\Pi$ be the diagonal matrix with $\pi$ on the diagonal. Let $H$ be the graph with adjacency matrix $\Pi W$. It can be verified that $H$ is Eulerian and has the same random walk matrix as $G$. 
\subsection{The Symmetric Lift}
A standard tool in the study of directed graphs, which we will use often, is the \textbf{symmetric lift} (also known as the bipartite double cover in the case of undirected graphs); see \cite{brualdi1980bigraphs}. 

\begin{definition}
Let $G=(V,E, w)$ be a weighted digraph. The \textbf{symmetric lift} of $G$ is the undirected bipartite graph $sl(G)$ where $V(sl(G)) = V\times \{L, R\}$ and $w_{sl(G)}((v,L),(u,R)) = w_G(v,u)$.
\end{definition}

We can also define the symmetric lift of a matrix $B$ as $sl(B) = \begin{bmatrix}
    0 & B\\
    B^T & 0
\end{bmatrix}$.
 By inspection, the symmetric lift $sl(G)$ satisfies that $A_{sl(G)} = sl(A_G)$. The following lemma relates the spectrum of $A_{sl(G)}$ to the singular values of $A_G$.

\begin{lemma}\label{lem:eig_vals_sl}
    Let $G$ be a weighted directed graph with normalized adjancency matrix $A_G$. The eigenvalues of $A_{sl(G)}$ are $\pm\sigma_1(A_G), \pm\sigma_2(A_G), \ldots, \pm\sigma_n(A_G)$.
\end{lemma}
\begin{proof}
    Since $sl(G)$ is bipartite, the eigenvalues of $A_{sl(G)}$ are symmetric about $0$. To determine the eigenvalues of $A_{sl(G)}$ we consider its square. Recall that the eigenvalues of $A^2_{sl(G)}$ are the squares of the eigenvalues of $A_{sl(G)}$. By calculations we have, 
    \begin{center}
        \[A^2_{sl(G)} = \begin{bmatrix}
            0 & A \\
            A^T & 0
        \end{bmatrix}^2 =
       \begin{bmatrix}
            AA^T & 0 \\
            0 & A^TA
        \end{bmatrix} \].
    \end{center}
  This implies that the eigenvalues of $A^2_{sl(G)}$ are the eigenvalues of $AA^T$ and $A^TA$, but these matrices have the same eigenvalues. By definition, the eigenvalues of $AA^T$ are the squares of singular values of $A$. Since the eigenvalues of $A_{sl(G)}$ are symmetric about $0$, it follows that the eigenvalues of $A_{sl(G)}$ are $\pm\sigma_1(A_G), \pm\sigma_2(A_G), \ldots, \pm\sigma_n(A_G)$.
\end{proof}

\section{Cheeger's Inequality and Our Singular-Value Analogue}\label{sec_sv_cheeger}

\subsection{Cheeger's Inequality}
A basic fact in spectral graph theory is that for an undirected graph $G$ it holds that $\mu_2 = 1$ if and only if $G$ is disconnected. Informally, Cheeger's inequality shows that the second-largest eigenvalue of the normalized adjacency matrix of an undirected graph provides a robust measure of connectivity. It is a foundational result in theoretical computer science and discrete mathematics with many applications such as bounding mixing time of markov chains and spectral partitioning of graphs. In particular, $1-\mu_2$ is closely related to a combinatorial quantity called conductance.

\begin{definition}\label{def: conductance}
Let $G = (V,E)$ be an Eulerian digraph and $S\subseteq V$. The \textbf{conductance} of $S$ denoted $\phi(S)$ is defined by
   \[
\phi(S) = \frac{e(S, S^c)}{\vol(S)}.
\]
The \textbf{conductance} of $G$ is $\phi(G)$ defined by
   \[
\phi(G) = \min_{\substack{\emptyset \neq S\subseteq V;\\ \vol(S)\leq \vol(V)/2}}\phi(S).
\]
\end{definition}

Note that we let $G'$ is the ``undirectification" of $G$ (i.e. $G'$ has edge weights $w'(u,v) = ((w(u,v)+w(v,u))/2$), then $\phi(G') = \phi(G)$ for all $S$, so conductance does not capture directed information about $G$. Here we give a formal statement of Cheeger's inequality. 
\begin{theorem}[Cheeger's Inequality, cf. \cite{AlonEigenvaluesExpanders}]\label{thm:cheeger}
Let $G$ be an undirected graph and $1 = \mu_1 \geq \mu_2 \geq \ldots \geq \mu_n\geq -1$ be the eigenvalues of its normalized adjacency matrix $A_G$. Then,
\[
        \frac{1-\mu_2}{2}\leq \phi(G)\leq \sqrt{2(1-\mu_2)}
\]
\end{theorem}

\subsection{Our Singular Value Cheeger Inequality}
In this section, we prove a Cheeger-like inequality for $\sigma_2$, the second-largest singular value of the normalized adjacency matrix, for weighted Eulerian digraphs. In particular, we show that $\sigma_2$ robustly captures how close a graph is to having two sets $S,T$ such that all edges leaving $S$ enter $T$ and all edges entering $T$ come from $S$. \par For undirected graphs, note that $\sigma_2 = \max\{\mu_2, |\mu_n|\} = \min\{1-\mu_2, 1+\mu_n\}$. Thus, $\sigma_2$ unifies the spectral quantites arising in Theorem~\ref{thm:cheeger} and Theorem~\ref{thm:bipartite_cheeger}. 

First, we introduce our \emph{directed} notion of conductance and state the result formally.

\begin{definition}\label{def:phisl}
    Let $G = (V,E)$ be an Eulerian digraph and $S,T\subseteq V$. The \textbf{directed conductance} of $S$ and $T$ denoted $\phid(S,T)$ is defined by
   \[
\phid(S,T) = \frac{e(S, T^c) + e(S^c, T)}{\vol(S)+\vol(T)}.
\]
The \textbf{directed conductance} of $G$ is $\phid(G)$ defined by
   \[
\phid(G) = \min_{\substack{S,T\subseteq V; \\ 0<\vol(S)+\vol(T)\leq \vol(V)}}\frac{e(S, T^c) + e(S^c, T)}{\vol(S)+\vol(T)}.
\]
\end{definition}

See Figure~\ref{fig:phidir} for a visualization of $\phid(S,T)$. Note that for undirected graphs $\phi(S) = \phid(S,S)$.

Our Cheeger inequality is the following.

\begin{theorem}\label{thm:di-cheeger}
    Let $G$ be an Eulerian digraph and $\sigma_2$ the second-largest singular value of its normalized adjacency matrix $A_G$. Then,
\[
    \frac{1-\sigma_2}{2}\leq \phid(G) \leq \sqrt{2(1-\sigma_2)}.
\]
\end{theorem}

\noindent Like Cheeger's inequality, both inequalities are tight up to constant factors, as illustrated by the hypercube with self-loops and the undirected odd cycle respectively. Note that we need $G$ to be non-bipartite, else $\sigma_2 = 1$.

We prove Theorem~\ref{thm:di-cheeger} by applying the standard Cheeger inequality to $sl(G)$. To that end, we a define the projection of a subset of vertices in $sl(G)$ onto $V(G)$ and prove a lemma relating the cuts in $sl(G)$ and edges in the graph $G$. Given a finite set $X$, let $\mathcal{P}(X)$ be the set of all subsets of $X$. For a subset $S\subseteq V(sl(G))$ define $\pi_L(S) = \{u\in V: (u,L)\in S\}$. Similarly, define $\pi_R(S) = \{u\in V: (u,R)\in S\}$. 

\begin{lemma}\label{lem:relating_edges_vol_in_G_slG}
    Let $G = (V, E)$ be an Eulerian digraph and $sl(G)$ be the symmetric lift of $G$. Then the map $\pi:\mathcal{P}(V(sl(G)))\to \mathcal{P}(V)^2$ defined by $\pi(S) = (\pi_{L}(S), \pi_{R}(S))$ defines a bijection between subsets of $V(sl(G))$ and pairs of subsets of $V$ such that 
    \[\vol_{sl(G)}(S) = \vol_G(\pi_{L}(S)) + \vol_G(\pi_{R}(S)),\]
    and 
    \[e_{sl(G)}(S, S^c) = e_G(\pi_{L}(S), \pi_{R}(S)^c) + e_G(\pi_{L}(S)^c, \pi_{R}(S)).\]
\end{lemma}

\begin{proof}
  Let $G = (V,E)$ be an Eulerian digraph, $sl(G)$ be the symmetric lift of $G$ and $S\subseteq V(sl(G))$. By inspection, $\pi$ is a bijection.
  Now, we prove $\vol_{sl(G)}(S) = \vol_G(\pi_{L}(S_L)) + \vol_G(\pi_{R}(S_R))$. Recall, that $\vol_{sl(G)}(S) = \sum_{x\in S}d_{sl(G)}(x)$. Since $sl(G)$ is bipartite $S\cap L$ and $ S\cap R$ form a partition of $S$. Observe that for $x\in L$ we have $d_{sl(G)}(x) = d^+_{G}(\pi_{L}(x))$ and for $x\in R$ we have $d_{sl(G)}(x) = d^-_{G}(\pi_{R}(x))$. This allows us to write 
  \begin{align*}
      \vol_{sl(G)}(S) &= \sum_{x\in S\cap L}d_{sl(G)}(x) + \sum_{x\in S\cap R}d_{sl(G)}(x)\\
      &=\sum_{x\in S\cap L}d^+_{G}(x) + \sum_{x\in S\cap R}d^-_{G}(x)\\
      &= \vol_G(\pi_{L}(S\cap L)) + \vol_G(\pi_{R}(S\cap R)),
  \end{align*}
  where the last line follows from the definition of volume and $G$ being Eulerian.
  It remains to prove $e_{sl(G)}(S, S^c) = e_G(\pi_{L}(S), \pi_{R}(S)^c) + e_G(\pi_{L}(S)^c, \pi_{R}(S)).$  Since $sl(G)$ is bipartite, 
  \begin{align*}
  e_{sl(G)}(S, S^c) &= e_{sl(G}(S\cap L, R\setminus (R\cap S)) + e_{sl(G}(L\setminus (S\cap L), R\cap S)\\
  &= e_G(\pi_{L}(S), \pi_{R}(S)^c) + e_G(\pi_{L}(S)^c, \pi_{R}(S)),
  \end{align*}
  as desired.
\end{proof}

\begin{figure}[h]
    \centering
    \includegraphics[width=0.5\linewidth]{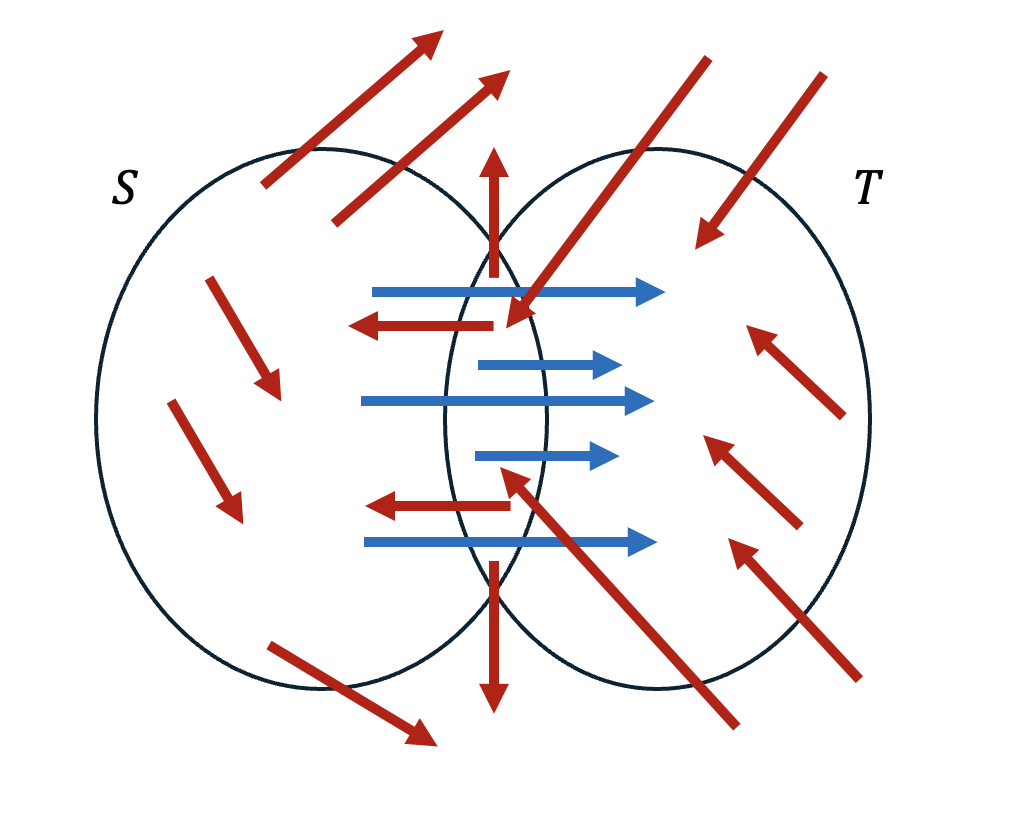}
    \caption{$\phid(S) = \frac{\text{red}}{\text{red} + 2\text{blue}}$}
    \label{fig:phidir}
\end{figure}

Now we prove a lemma relating the directed conductance of $G$ and the conductance of $sl(G)$. 
\begin{lemma}\label{lem:sl_cond_equals_cond_sl}
    Let $G = (V, E)$ be an Eulerian digraph and $sl(G)$ be the symmetric lift of $G$. Then $\phi(sl(G)) = \phid(G)$.
\end{lemma}

\begin{proof}
 We show $\phid(G) = \phi(sl(G))$. We begin with the upper bound $\phid(G)\leq \phi(sl(G))$. Let $S\subseteq V(sl(G))$ be such that $\phi(sl(G)) = \phi(S)$ and $\vol(S)\leq \vol(sl(G))/2 = \vol(G)$. We claim $\phid(\pi_L(S), \pi_R(S)) = \phi(S)$. By Lemma~\ref{lem:relating_edges_vol_in_G_slG}, we have $\vol(\pi_L(S)) + \vol(\pi_R(S)) = \vol(S)\leq \vol(G)$ and  $e_G(\pi_L(S), \pi_R(S)^c) + e_G(\pi_L(S)^c, \pi_R(S)) = e_{sl(G)}(S, S^c)$. Thus,
 \[
 \phid(G) \leq \phid(\pi_L(S), \pi_R(S)) = \phi(S) = \phi(sl(G)).
 \]
    
Next, we prove $\phi(sl(G))\leq \phid(G)$. Let $S_L,S_R\subseteq V$ with $\vol(S_L) + \vol(S_R)\leq \vol(G) = \vol(sl(G))/2$ be such that $\phid(S_L, S_R) = \phid(G)$. Define $S = \pi^{-1}(S_L,S_R)$. By Lemma~\ref{lem:relating_edges_vol_in_G_slG}, $\vol(S) =\vol(S_L) + \vol(S_R) \leq \vol(G)$ and $e_G(S_L, S_R^c) + e_G(S_L^c, S_R) = e_{sl(G)}(S, S^c)$. Hence, 
\[
\phi(sl(G))\leq \phi(S) = \phid(S_L, S_R) = \phid(G).
\]
\end{proof}

Combining Theorem~\ref{thm:cheeger} and Lemma~\ref{lem:sl_cond_equals_cond_sl} we get a simple proof of Theorem~\ref{thm:di-cheeger}.

\begin{proof}[Proof of Theorem \ref{thm:di-cheeger}]
    Let $G = (V,E)$ be an Eulerian digraph and $sl(G)$ be th symmetric lift of $G$. By applying Theorem~\ref{thm:cheeger} to $sl(G)$ and using Lemma~\ref{lem:eig_vals_sl}, we have
    \[
    \frac{1-\sigma_2}{2} \leq \phi(sl(G))\leq \sqrt{2(1-\sigma_2)}.
    \]
    By Lemma~\ref{lem:sl_cond_equals_cond_sl}, 
    \[
    \frac{1-\sigma_2}{2} \leq \phid(G)\leq \sqrt{2(1-\sigma_2)},
    \]
    as desired.
\end{proof}

 In Theorem~\ref{thm:cheeger}, the lower bound is tight due to the hypercube and the upper bound is tight due to the cycle graph. For both Theorem~\ref{thm:di-cheeger} we can consider the graphs whose symmetric lift are the hypercube and cycle graph respectively. 

We end this section with a result showing that for regular graphs, in $\phid(G)$ the sets $S,T$ can be made to have the same size with a constant factor loss in $\phid(G)$.
\begin{definition}
   Let $G= (V,E)$ be an Eulerian digraph. We define the \textbf{balanced directed conductance} by
   \[
\phi^=_{dir}(G) = \min_{\substack{S,T\subseteq V; \\  |S|=|T|}}\phid(S,T).
\]
\end{definition}

\begin{proposition}\label{prop:balanced_sv_cheeger}
    Let $G$ be a regular digraph. Then $$\phid(G)\leq \phi^{=}_{dir}(G)\leq 2\phid(G)/(1-\phid(G)).$$
\end{proposition}

\begin{proof}
    Let $G$ be a $d$-regular digraph. The inequality $\phid(G)\leq \phi^{=}_{dir}(G)$ holds since $\phid(G)$ is minimizing the same quantity as $\phi^{=}_{dir}(G)$ over a larger collection of sets $S,T$.
    \par It remains to prove $\phi^{=}_{dir}(G)\leq 2\phid(G)/(1-\phid(G))$. Let $S, T\subseteq V(G)$ such that $\phid(S, T) = \phid(G)$. If $|S| = |T|$ we are done, so without loss of generality assume $\left|S\right| > \left|T\right|$. Let $S'$ be the set obtained from $S$ removing an arbitrary subset of $\left|S\right| - \left|T\right|\leq \phid(S, T)(|S|+|T|)$ vertices from $S$. We show $\phid(S', T)\leq 2\phid(S,T)/(1-\phid(S,T))$. We have $ d(|S|-|T|)\leq e(S, T^c)$ and   $e(S,T^c) \leq \phid(S,T)d(|S|+|T|)$. Thus, 
    \[
    |S-S'|=|S|-|T|\leq \phid(S,T)(|S|+|T|).
    \]

    Observe,
    \begin{align*}
        e(S', T^c) + e((S')^c, T) &= e(S, T^c)+e(S^c, T)-e(S\setminus S', T^c)+e(S\setminus S', T)\\
        &\leq e(S, T^c)+e(S^c, T)+e(S\setminus S', T)\\
        &\leq e(S, T^c)+e(S^c, T) + d\cdot|S\setminus S'|\\
        &\leq 2d\cdot\phid(S,T)(|S|+|T|)
    \end{align*}
    
    We also have 
    \begin{align*}
        \vol(S')+\vol(T) &= d(\left|S\right|+|T|-|S\setminus S'|)\\
        &\geq (1-\phid(S, T))\cdot d(\left|S\right|+\left|T\right|).
    \end{align*}
    Putting this together gives us $\phid(S', T) \leq 2\phid(S)/(1-\phid(S))$.
\end{proof}

\section{Relationship to Bipartiteness Cheeger}\label{sec:relation_to_bipartiteness_cheeger}

\subsection{Trevisan's Bipartiteness Cheeger}
Trevisan \cite{TrevisanMaxCut} proves a Cheeger inequality for $\mu_n$ and by way of this result, he designed a new efficient approximation algorithm for the max-cut problem. 

\begin{definition}[\cite{TrevisanMaxCut}]
    Let $G = (V,E)$ be an undirected graph. For $y\in\{-1,0,1\}^{n}$ the \textbf{bipartiteness ratio} of $y$ is defined by $$\beta(y) = \frac{\sum_{u,v}w_{u,v}\left|y_u+y_v\right|}{2\sum_{v}d_v\left|y_v\right|}$$. The bipartiteness ratio of $G$ is defined by 
    \[
        \beta(G) = \min_{y\in\{-1,0,1\}^n- 0^n}\beta(y).
    \]
\end{definition}

Let $S = \{v\in V: y_v \neq 0\}, A = \{v\in V: y_v = 1\}$ and $B = \{v\in V: y_v = -1\}$. The quantity $\beta(y)$ captures the fraction of edges incident to $S$ that must be removed to make $S$ bipartite with bipartition $(A,B)$, as demonstrated by the following lemma.

\begin{lemma}[\cite{TrevisanMaxCut}]\label{lem:bipartite_ratio}
    Let $G = (V,E)$ be an undirected graph. For every $y\in \{-1,0,1\}^n$, define $S = \{v\in V: y_v \neq 0\}, A = \{v\in V: y_v = 1\}$ and $B = \{v\in V: y_v = -1\}$. Then,
    \[\beta(y) = \frac{e(A)+e(B)+e(S, S^c)}{\vol(S)}\]
\end{lemma}

\begin{proof}
    We prove the lemma by showing $\sum_{v}d_v|y_v| = \vol(S)$ and $\sum_{u,v}w_{u,v}|y_u+y_v|/2 = e(A)+e(B)+e(S, S^c)$. The first equality follows from the definition of $S$ and $\vol(S)$. For the second we have
    \begin{align*}
      \frac{1}{2}\sum_{u,v}w_{u,v}|y_u+y_v| &= \frac{1}{2}\sum_{u,v\in A}2w_{u,v} + \frac{1}{2}\sum_{u,v\in B}\sum_{v\in B}2w_{u,v} + \frac{1}{2}\sum_{u\in S, v\in S^c}w_{u,v} + \frac{1}{2}\sum_{u\in S^c, v\in S}w_{u,v}\\
      &= e(A) + e(B) + e(S, S^c).
    \end{align*}
\end{proof}

Trevisan establishes the following quantitative relationship between $\mu_n$ and $\beta(G)$.

\begin{theorem}[Bipartiteness Cheeger, \cite{TrevisanMaxCut}]\label{thm:bipartite_cheeger}
Let $G$ be a $d$-regular graph and $1 = \mu_1 \geq \mu_2 \geq \ldots \geq \mu_n\geq -1$ be the eigenvalues of its normalized adjacency matrix $A_G$. Then,
\[
 \frac{1+\mu_n}{2} \leq \beta(G) \leq \sqrt{2(1+\mu_n)}.
\]
\end{theorem}

\subsection{A Directed Bipartiteness Ratio}
In this section we prove that for undirected graphs, we are able to show that $\phid(G)$ essentially captures both the ordinary conductance $\phi(G)$ and the bipartite ratio $\beta(G)$. This can be viewed as a combinatorial version of the fact that $\sigma_2 = \max\{\mu_2, |\mu_n|\} = \min\{1-\mu_2, 1+\mu_n\}$. Furthermore, we show that an analogous result for directed graphs cannot hold. We begin by relating conductance and our directed conductance.

\begin{lemma}\label{lem:phi_dirphi}
    Let $G = (V,E)$ be an Eulerian digraph. For every $S\subseteq V$, $\phi(S) = \phid(S,S)$. In particular, 
    \[\phi(G) = \min_{\substack{S\subseteq V;\\  0<\vol(S)\leq \vol(V)/2}}\phid(S,S). \]
\end{lemma}

For a bipartiteness analogue of Lemma~\ref{lem:phi_dirphi}, instead of minimizing over $S=T$ as above, we minimize over $S$ disjoint from $T$.

\begin{definition}\label{def: phibp}
Let $G = (V,E)$ be an Eulerian digraph. The \textbf{directed bipartiteness ratio} of $G$ is $\bdir(G)$ defined by
   \[
\bdir(G) = \min_{\substack{S,T\subseteq V;\\ S\cap T=\emptyset; \\ 0<\vol(S)+\vol(T)\leq \vol(V)}}\phid(S,T)
\]
\end{definition}

This is a directed generalization of Trevisan's bipartiteness ratio as shown by the following lemma:

\begin{lemma}\label{lem:phibp_vs_beta}
If $G = (V,E)$ is an undirected graph, then $\bdir(G) = \beta(G)$
\end{lemma}

\begin{proof}
    Let $G$ be an undirected graph. Let $y\in\{-1,0,1\}$ and define $A = \{v\in V: y_v = 1\}, B = \{v\in V: y_v = -1\}$ and $S = A\cup B$. We show that $\bdir(A, B) = \beta(y)$. By Lemma~\ref{lem:bipartite_ratio}, we have $\beta(y) = (e(A)+e(B)+e(S,S^c))/\vol(S)$. Observe that we can write $\bdir(A,B)$ as follows
    \begin{align*}
        \bdir(A,B) &= \frac{e(A, B^c)+e(A^c,B)}{\vol(A)+\vol(B)}\\
        &=\frac{e(A) + e(A, S^c) + e(B) + e(S^c, B)}{\vol(S)}\\
        &= \frac{e(A)+e(B) + e(S, S^c)}{\vol(S)} &&\text{(by undirectedness)}\\
        &= \beta(y)
    \end{align*}
    Taking the minimum over all nonzero $y\in \{-1, 0 ,1\}$ gives the desired equality.
\end{proof}

\subsection{Directed Conductance vs Undirected Conductance and Bipartiteness}
Now we prove our main result of the section, namely that for undirected graphs, $\phid(G)$ is within a constant factor of $\min\{\phi(G), \bdir(G)\}$.

\begin{figure}[H]
\centering
\begin{subfigure}[t]{0.4 \textwidth}
\centering
\scalebox{.6}{
\begin{tikzpicture}[
    -,>=stealth',shorten >=1pt,
    auto,node distance=4cm,
    thick,
    main node/.style={circle,draw,font=\small\bfseries}
  ]
  \node[main node] (1) {$S\cap T$};
  \node[main node] (2) [below left  of=1] {$S\setminus T$};
  \node[main node] (3) [below right of=1] {$T\setminus S$};
  \node[main node] (4) [below       of=1] {$S^{c}\!\cap T^{c}$};

  \draw[blue] (1)--(2) (1)--(3);
  \draw[blue] (1) to [bend left] (4);

  \draw (2) to[out=200,in=150,looseness=8] (2);
  \draw (3) to[out=340,in= 30,looseness=8] (3);
  \draw (2)--(4) (3)--(4) (4)--(1);

  \begin{scope}[on background layer,dashed,thick]
    \coordinate (Sc) at ($(1)!0.5!(2)$);           
    \draw (Sc) ellipse [x radius=2.8cm,
                        y radius=1.4cm,
                        rotate= 45]               
          node at ($(Sc)+(-2.2, 1.2)$) {$S$};

    \coordinate (Tc) at ($(1)!0.5!(3)$);           
    \draw (Tc) ellipse [x radius=2.8cm,
                        y radius=1.4cm,
                        rotate=-45]               
          node at ($(Tc)+( 2.2, 1.2)$) {$T$};
  \end{scope}
\end{tikzpicture}
}
\caption{The blue edges correspond to the edges counted by $e(S\cap T, (S\cap T)^c)$.}
\end{subfigure}
~
\begin{subfigure}[t]{0.4 \textwidth}
\centering
\scalebox{.6}{
\begin{tikzpicture}[
    -,>=stealth',shorten >=1pt,
    auto,node distance=4cm,
    thick,
    main node/.style={circle,draw,font=\small\bfseries}
  ]
  \node[main node] (1) {$S\cap T$};
  \node[main node] (2) [below left  of=1] {$S\setminus T$};
  \node[main node] (3) [below right of=1] {$T\setminus S$};
  \node[main node] (4) [below of=1] {$S^{c}\!\cap T^{c}$};

  \draw[red] (1)--(2) (1)--(3) (4)--(2) (4)--(3);

  \draw[red] (2) to[out=200,in=150,looseness=8] (2);
  \draw[red] (3) to[out=340,in= 30,looseness=8] (3);
  \draw (2)--(4) (3)--(4) (4)--(1);

  \begin{scope}[on background layer,dashed,thick]
    \coordinate (Sc) at ($(1)!0.5!(2)$);           
    \draw (Sc) ellipse [x radius=2.8cm,
                        y radius=1.4cm,
                        rotate= 45]               
          node at ($(Sc)+(-2.2, 1.2)$) {$S$};

    \coordinate (Tc) at ($(1)!0.5!(3)$);           
    \draw (Tc) ellipse [x radius=2.8cm,
                        y radius=1.4cm,
                        rotate=-45]               
          node at ($(Tc)+( 2.2, 1.2)$) {$T$};
  \end{scope}
\end{tikzpicture}
}
\caption{The red edges correspond to the edges counted by $e(S\setminus T, (T\setminus S)^c)$ and $e((S\setminus T)^c, T\setminus S)$.}
\end{subfigure}
\caption{Visualization of the two cases in the proof of Theorem~\ref{thm:relating_phi_phisl_phibp}.}
\end{figure}
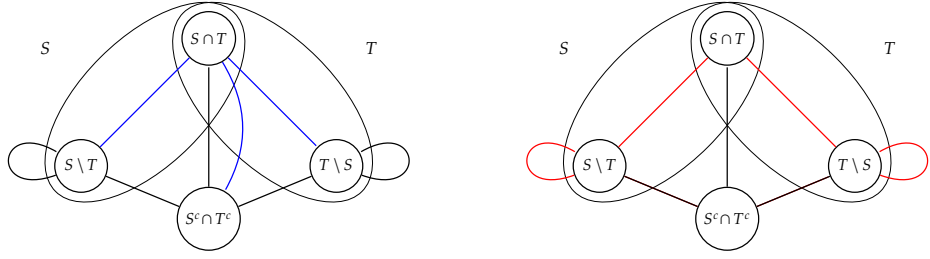\label{fig:edge_removal_visualization_2}

\begin{theorem}\label{thm:relating_phi_phisl_phibp}
Let $G$ be an Eulerian graph. Then,
 \[\phid(G) \leq \min\{\phi(G), \bdir(G)\}.\]
If $G$ is undirected, then we also have
\[ \phid(G)\geq \min\{\phi(G), \beta(G)\}/3.\]
\end{theorem}

Before proving Theorem~\ref{thm:relating_phi_phisl_phibp}, we prove the following lemma. 
\begin{lemma}\label{lem:relating_phi_phisl_phibp}
    Let $G$ be an undirected graph and $S,T\subseteq V$ such that $0 <\vol(S)+\vol(T)$. Then,
    \[
    \min\{\phi(S\cap T), \beta(S\setminus T, T\setminus S)\}\leq 3\phid(S,T)
    \]
\end{lemma}
 
\begin{proof}
The proof consists of two cases.\\
\textbf{Case 1} $\vol(S\cap T)\geq (\vol(S)+\vol(T))/3$: In this case we show $\phi(S\cap T)\leq 3\phid(S, T)$. We do this by showing $e(S\cap T, (S\cap T)^c)\leq e(S, T^c) + e(S^c, T)$. See Figure~\ref{fig:edge_removal_visualization_2}(a) for a visualization. Note that we can write 
\[e(S\cap T, (S\cap T)^c) = e(S\cap T, S\setminus T) + e(S\cap T, T\setminus S) + e(S\cap T,(S\cup T)^c).\]
Observe that the terms in $e(S\cap T, S\setminus T)$ are counted in $e(S, T^c)$, the terms in $e(S\cap T, T\setminus S)$ are counted in $e(S^c, T)$ and the terms in $e(S\cap T,(S\cup T)^c)$ are counted in both $e(S^c, T)$ and $e(S, T^c)$. This proves the inequality, as desired. Combining our assumption that $(\vol(S)+\vol(T))/3\leq \vol(S\cap T)$ with the inequality $e(S\cap T, (S\cap T)^c)\leq e(S, T^c) + e(S^c, T)$, we obtain
\[\phi(G)\leq \phi(S\cap T) = \frac{e(S\cap T, (S\cap T)^c}{\vol(S\cap T)} \leq \frac{e(S,T^c) + e(S^c, T)}{(\vol(S)+\vol(T))/3} \leq 3\phid(S,T).\]

\textbf{Case 2} $\vol(S\cap T)\leq (\vol(S)+\vol(T))/3$: Note that $\vol(S)+\vol(T) = 2\vol(S\cap T) + \vol(S\setminus T) + \vol(T\setminus S)$. The assumption that $\vol(S\cap T)\leq (\vol(S)+\vol(T))/3$ thus implies that $\vol(S\setminus T) + \vol(T\setminus S) \geq (\vol(S)+\vol(T))/3$. In this case, we show that $\phid(S\setminus T, T\setminus S)\leq 3\phid(S,T)$. To prove this, we show 
\[e(S\setminus T, (T\setminus S)^c))+e((S\setminus T)^c, T\setminus S)\leq e(S, T^c) + e(S^c, T).\]
See Figure~\ref{fig:edge_removal_visualization_2}(b) for a visualization. Observe that $e(S\setminus T, (T\setminus S)^c))\leq e(S, T^c)$ and $e((S\setminus T)^c, T\setminus S)\leq e(S^c, T)$. This implies,
\[\phid(S\setminus T, T\setminus S)\leq 3\phid(S,T),\]
as desired.
\end{proof}

\begin{proof}[Proof of Theorem~\ref{thm:relating_phi_phisl_phibp}]
Let $G = (V,E)$ be an Eulerian graph.  We begin with the upper bound, $\phid(G)\leq \min\{\phi(G), \bdir(G)\}$. This holds because $\phi(G)$ is the minimum of $\phi(S) = \phid(S,S)$ over sets $S$ of volume at most $\vol(V)/2$ and $\bdir(G)$ is the minimum of $\phid(S,T)$ over disjoint sets $S, T$ such that $\vol(S)+\vol(T)\leq \vol(V)$.
\par Now assume that $G$ is undirected. We show that we can bound $\min\{\phi(G), \beta(G)\}$ above by $3\phid(G)$. Let $S,T\subseteq V$ such that $\phid(S,T) = \phid(G)$ and $\vol(S) + \vol(T) \leq \vol(V)$. Applying Lemma~\ref{lem:relating_phi_phisl_phibp}, we have \[\min\{\phi(G), \beta(G)\}\leq \min\{\phi(S\cap T), \beta(S\setminus T, T\setminus S)\}\leq 3\phid(S,T) = 3\phid(G),\]
which completes the proof.
\end{proof}


The loss of a factor of $3$ between $\min\{\phi(G), \beta(G)\}$ and $\phid(G)$ in the proof of Theorem~\ref{thm:relating_phi_phisl_phibp} comes from the fact that we are removing elements from $S$ and $T$, which affects the volume. 

Now we show the second inequality in Theorem~\ref{thm:relating_phi_phisl_phibp} cannot hold for all Eulerian (or even regular) directed graphs even if we replace $3$ with any finite quantity even a function of $n$.

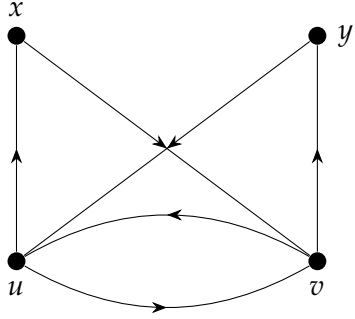
\begin{figure}[H]
\centering
\begin{tikzpicture}[
       decoration = {markings,
                     mark=at position .5 with {\arrow{Stealth[length=2mm]}}},
       dot/.style = {circle, fill, inner sep=2.4pt, node contents={},
                     label=#1},
every edge/.style = {draw, postaction=decorate}
                        ]
\node (x) at (0,5) [dot=$x$];
\node (u) at (0,2) [dot=below:$u$];
\node (v) at (4,2) [dot=below:$v$];
\node (y) at (4,5) [dot=right:$y$];
\path   (u) edge[bend right] (v)    (v) edge[bend right] (u)    
        (u) edge (x)     (x) edge (v)
        (v) edge (y)     (y) edge (u);
\end{tikzpicture}
\caption{An Eulerian graph $G$ such $\phid(G) = 0$ and $\min\{\phi(G), \bdir(G)\} > 0$.} \label{fig:counter_example}
\end{figure}

\begin{proposition}\label{prop:counter_example}
    There exists a regular directed graph $G$ such that 
\begin{center}
    $\phid(G) = 0$ and $\min\{\phi(G), \bdir(G)\} > 0$
\end{center}
\end{proposition}

\begin{proof}
    We begin by constructing an Eulerian graph and then show how to make it regular. Let $c > 0$ and $G = (V,E)$ be defined by $V = \{x,y,u,v\}$ and $E = \{(u,x), (u,v), (v, u), (v, y), (x, v), (y,u)\}$. See Figure~\ref{fig:counter_example} for a picture of $G$. Note that $G$ is Eulerian. Consider $S = \{x, v\}$ and $T = \{x, u\}$. By inspection, $\vol(S) = \vol(T) = 3\leq \vol(V)/2$ and $e(S, T^c)+e(S^c,T) = 0$, and so $\phid(S,T) = \phid(G) = 0.$ To complete the proof, we must show $0<\min\{\phi(G), \bdir(G)\}$. 
    
    \par First we show $0 < \phi(G)$. Let $S\subseteq V$ be nonempty and satisfy $\vol(S)\leq \vol(V)/2 = 3$. This implies that $S$ cannot contain both $u$ and $v$ and $|S|\leq 2$. If $|S| = 1$, then $e(S, S^c) + e(S^c, S) > 0$, and so $\phi(S) > 0$. If $|S| = 2$, then $S$ contain a vertex in $\{x,y\}$. By symmetry, we may assume $x\in S$. No matter which other vertex is in $S$, we have $e(S, S^c) + e(S^c, S) > 0$. Thus, $\phi(G) > 0$.
    
   \par Next, we show $\bdir(G) > 0$. Let $S, T\subseteq V$ satisfy $0 <\vol(S) + \vol(T)\leq \vol(V)$ and $S\cap T = \emptyset$. Note, $S, T, (S\cup T)^c$ partitions $V$. It suffices to show $e(S, T^c) + e(S^c, T) > 0$. If either $S$ or $T$ is empty, then $e(S, T^c) + e(S^c, T) > 0$. Hence, we may assume both $S$ and $T$ are non-empty. By the pigeon-hole principle, $S, T$, or $(S\cup T)^c$ must contain $2$ vertices. First assume $S$ or $T$ has at least $2$ vertices. By symmetry, we may assume $|S| = 2$. In this case, $e(S) > 0$ unless $S = \{x,y\}$, and so $e(S, T^c) > 0$. Otherwise $S = \{x,y\}$, and either $T = \{u,v\}$ or either $u$ or $v$ is in  $(S\cup T)^c$. If $T = \{u,v\}$, then $e(S^c, T) > 0$. If either $u$ or $v$ is in  $(S\cup T)^c$, then $e(S, T^c) > 0$. The final case is $|(S\cup T)^c| = 2$. If $S = \{u\}$ or $\{v\}$, then one of their two out-neighbors must be in  $(S\cup T)^c$. This implies, $e(S, T^c) > 0$. Otherwise, $S = \{x\}$ or $S = \{y\}$. By symmetry we assume $S = \{x\}$. Now $e(S, T^c) > 0$ unless, $T = \{v\}$. But, in this case, $u\in (S\cup T)^c$. This implies $e(S^c, T) > 0$. This shows $\bdir(G) > 0$ and completes the proof.

   \par To obtain a regular graph we can duplicate the vertices of degree $2$ (namely $u$ and $v$) and split the edge weights between the copies to obtain the following graph. Let $H = (V,E)$ be defined by $V(H) = \{x,y,u,v\}$ and $E(H) = \{(u_1,x), (u_1,v_1), (v_1, u_1), (v_1, y), (x, v_1), (y,u_1), (u_2,x), (u_2,v_2), (v_2, u_2), \\(v_2, y), (x, v_2), (y,u_2)\}$. See Figure~\ref{fig:counter_example2}
\end{proof}

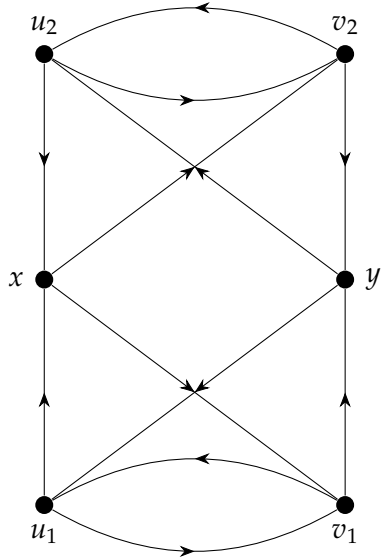
\begin{figure}[H]
\centering
\begin{tikzpicture}[
       decoration = {markings,
                     mark=at position .5 with {\arrow{Stealth[length=2mm]}}},
       dot/.style = {circle, fill, inner sep=2.4pt, node contents={},
                     label=#1},
every edge/.style = {draw, postaction=decorate}
                        ]
\node (x) at (0,5) [dot=left:$x$];
\node (u_1) at (0,2) [dot=below:$u_1$];
\node (u_2) at (0,8) [dot=above:$u_2$];
\node (v_1) at (4,2) [dot=below:$v_1$];
\node (v_2) at (4,8) [dot=above:$v_2$];
\node (y) at (4,5) [dot=right:$y$];
\path   (u_1) edge[bend right] (v_1)    (v_1) edge[bend right] (u_1)    
        (u_1) edge (x)     (x) edge (v_1)
        (v_1) edge (y)     (y) edge (u_1);

 \path (u_2) edge[bend right] (v_2)    (v_2) edge[bend right] (u_2)    
        (u_2) edge (x)     (x) edge (v_2)
        (v_2) edge (y)     (y) edge (u_2);   
\end{tikzpicture}

\caption{A regular graph $G$ such $\phid(G) = 0$ and $\min\{\phi(G), \bdir(G)\} > 0$.} \label{fig:counter_example2}
\end{figure}

\section{Higher-Order Cheeger Inequalities}\label{sec:higher_order_sv}
In this section we prove higher-order Cheeger like inequality for the $k$-th largest singular value of the normalized adjacency matrix $\sigma_k$ for Eulerian digraphs. In particular, we show that $\sigma_k$ robustly captures how close a graph is to having two families of sets $S_1, \ldots, S_k$ and $T_1,\ldots, T_k$ such that all edges leaving $S_i$ enter $T_i$ and all edges entering $T_i$ come from $S_i$. We begin by introducing necessary definitions, then state results formally. 

\subsection{The Higher-Order Cheeger Inequalities}
Like Cheeger's inequality, the Higher-Order Cheeger Inequalities show that $\mu_k$ and $\mu_n$ robustly measure other combinatorial quantities. In particular, they measure how close graph is to having $k$ connected components and a bipartite component, respectively, generalizing the fact that for a graph $G$ and $\mu_k = 1$ if and only if $G$ has at least $k$ components, and $\mu_n = -1$ if and only if $G$ has a bipartite component.

\begin{definition}
    For a positive integer $k$, we define the \textbf{$k$-way expansion constant} to be
    \[
        \rho_k(G) = \min_{S_1,\ldots, S_k}\max\{\phi_{G}(S_i)\colon i\in[k]\},
    \]
     where the min is over all collections of $k$ non-empty, disjoint subsets $S_1,\ldots, S_k\subseteq V$. 
\end{definition}

The higher-order Cheeger inequalities of Lee, Oveis Gharan, and Trevisan \cite{higher-order-cheeger} shows that $\mu_k$ robustly measures how close a graph is to having $k$-connected components.

\begin{theorem}[higher-order Cheeger, \cite{higher-order-cheeger}]\label{thm:higher_order_cheeger}
    Let $G$ be an undirected graph, $1 = \mu_1 \geq \mu_2 \geq \ldots \geq \mu_n\geq -1$ be the eigenvalues of its normalized adjacency matrix $A_G$ and $k\in\mathbb{N}$, we have
    \[
        \frac{1-\mu_k}{2} \leq \rho_k(G)\leq O(k^2)\cdot\sqrt{(1-\mu_k)}.
    \]
\end{theorem}

This result resolved a conjecture of Miclo \cite{miclo2008eigenfunctions} and led to new algorithms for $k$-way spectral partitioning.

\subsection{Our Singular-Value Higher Order Cheeger Inequality}
In this section we prove higher-order Cheeger like inequality for the $k$-th largest singular value of the normalized adjacency matrix $\sigma_k$ for Eulerian digraphs. In particular, we show that $\sigma_k$ robustly captures how close a graph is to having two families of sets $S_1, \ldots, S_k$ and $T_1,\ldots, T_k$ such that all edges leaving $S_i$ enter $T_i$ and all edges entering $T_i$ come from $S_i$. We begin by introducing necessary definitions, then state results formally. 

\begin{definition}\label{def:k-way-sl-cond}
    Let $G$ be an Eulerian graph and $k\in[n]$. The \textbf{$k$-way directed conductance} is defined by 
    \[
     \phi_{k,dir}(G) = \min_{\substack{S_1, \ldots, S_k; \\ T_1,\ldots, T_k}}\max\{\phid(S_i, T_i): i\in[k]\},
    \]  
where the min is over all collections of non-empty subsets $S_1, \ldots, S_k$ and $T_1,\ldots, T_k$ of $V$ such that $S_i\cap S_j =T_i\cap T_j=\emptyset$ for all $i\neq j\in[k]$. 
\end{definition}

For Eulerian digraphs we obtain the following higher-order Cheeger inequality for singular values.

\begin{theorem}\label{thm:sv_higher_order_cheeger}
    Let $G$ be an Eulerian directed graph and $k\in\mathbb{N}$. Then
    \[
      \frac{1-\sigma_k}{2} \leq 
        \phi_{k, dir}(G)\leq O(k^2)\cdot\sqrt{1-\sigma_k}.
    \] 
\end{theorem}

Analogous to Theorem~\ref{thm:relating_phi_phisl_phibp}, for undirected graphs we prove a stronger result showing that in $\phi_{k,dir}(G)$ the sets $S_1,\ldots, S_k$ and $T_1, \ldots, T_k$ can be made to satisfy either $S_i\cap T_i = \emptyset$ or $S_i=T_i$ for all $i\in[k]$.

\begin{definition}
    Let $G = (V,E)$ be an Eulerian directed graph and $k\in\mathbb{N}$. We define $\rho_{k,dir}(G)$ by 
    \begin{align*}
        \rho_{k,dir}(G) = \min_{\substack{S_1, \ldots, S_k; \\ T_1,\ldots, T_k}}\max\{\phid(S_i, T_i): i\in[k]\},
    \end{align*}
where the min is over all collections of non-empty subsets $S_1, \ldots, S_k$ and $T_1,\ldots, T_k$ of $V$ such that for all $i\neq j\in [k]$ we have $S_i\cap S_j = T_i\cap T_j = \emptyset$ and for all $i\in[k]$ we have $S_i = T_i$ or $S_i\cap T_i = \emptyset$.
\end{definition}

For undirected graphs show that $\rho_{k,dir}$ is within a  constant factor of the k-way directed conductance $\phi_{k,dir}$ and thus also has a relation to $\sigma_k$ like in Theorem~\ref{thm:sv_higher_order_cheeger}. So, we can use the former in our higher-order Cheeger for singular values.

\begin{theorem}\label{thm:und_sv_higher-order-cheeger}
    Let $G$ be an undirected graph and $k\in\mathbb{N}$. The $k$-way directed conductance of $G$ satisfies
    \begin{align*}
       \phi_{k,dir}(G)\leq \rho_{k,dir}(G)\leq 3\cdot\phi_{k,dir}(G).
    \end{align*}
\end{theorem}

\begin{corollary}
     Let $G$ be an undirected graph and $k\in\mathbb{N}$. Then,
     \[
     \frac{1-\sigma_k}{2}\leq \rho_{k,dir}\leq O(k^2)\sqrt{1-\sigma_k}.
     \]
\end{corollary}

Before proving Theorem~\ref{thm:sv_higher_order_cheeger} and Theorem~\ref{thm:und_sv_higher-order-cheeger} we prove a result relating $\sigma_k$ to combinatorial properties of graphs.

The Lemma below relates the $k$-way directed conductance of an Eulerian digraph $G$ to the $k$-way expansion of the symmetric lift of $G$.
\begin{lemma}\label{lem:higher_order_sl_cond_equals_cond_sl}
    Let $G = (V,E)$ be an Eulerian digraph and $sl(G)$ be the symmetric lift of $G$. Then $\phi_{k,dir}(G) = \rho_k(sl(G))$.
\end{lemma}

\begin{proof}
    We show that $\phi_{k,dir}(G) = \rho_{k}(sl(G))$. First, we prove $\phi_{k,dir}(G) \leq \rho_{k}(sl(G))$. Let $X_1, \ldots, X_k$ be non-empty disjoint subsets of $V(sl(G))$ such that $ \rho_{k}(sl(G)) = \max\{\phi_{sl(G)}(X_i): i\in[k]\}$. Let $S_i = \pi_{L}(X_{i})$ and $T_i = \pi_{R}(X_i)$. Since the $X_i$'s are pairwise disjoint, $S_i\cap S_j = T_i\cap T_j = \emptyset$ for $i\neq j$. By Lemma~\ref{lem:sl_cond_equals_cond_sl}, $\phi_{sl(G)}(X_i) = \phi_{dir}(S_i, T_i)$. Thus, 
    \[\phi_{k,dir}(G) \leq \max\{\phid(S_i,T_i):i\in[k]\} = \max\{\phi_{sl(G)}(X_i): i\in[k]\} = \rho_k(sl(G)).\]
    \par It remains to show $\rho_{k}(sl(G))\leq \phi_{k,dir}(G)$. Let $S_1,\ldots, S_k$ and $T_1, \ldots, T_k$ be non-empty subsets of $V(G)$ such that for all $i\neq j$, $S_i \cap S_j = T_i\cap T_j = \emptyset$ and $\phi_{k,dir}(G) = \max\{\phid(S_i,T_i):i\in[k]\}$. For each $i\in [k]$, let $X_i = \pi^{-1}(S_i, T_i)$. By Lemma~\ref{lem:sl_cond_equals_cond_sl}, for all $i\in[k]$ we have $\phid(S_i, T_i) = \phi_{sl(G)}(X_i)$. Hence, 
    \[\rho_{k}(sl(G)) \leq \max\{\phi_{sl(G)}(X_i):i\in[k]\} = \max\{\phid(S_i, T_i):i\in[k]\} = \phi_{k,dir}(G),\]
    as desired.
\end{proof}

Combining Lemma~\ref{lem:higher_order_sl_cond_equals_cond_sl} with Theorem~\ref{thm:higher_order_cheeger} we prove Theorem~\ref{thm:sv_higher_order_cheeger}.

\begin{proof}[Proof of Theorem \ref{thm:sv_higher_order_cheeger}]
    Let $G = (V,E)$ be an Eulerian graph and $k\in\in[n]$. Let $sl(G)$ be the symmetric lift of $G$. By Theorem~\ref{thm:higher_order_cheeger},
        \[
      (1-\sigma_k)/2 \leq 
        \rho_{k}(sl(G))\leq O(k^2)\sqrt{1-\sigma_k}.
    \] 
Applying Lemma~\ref{lem:higher_order_sl_cond_equals_cond_sl},
        \[
      (1-\sigma_k)/2 \leq 
        \phi_{k,dir}(G)\leq O(k^2)\sqrt{1-\sigma_k},
    \] 
 as desired.
\end{proof}

We conclude the section with the proof of Theorem~\ref{thm:und_sv_higher-order-cheeger}.

\begin{proof}[Proof of Theorem~\ref{thm:und_sv_higher-order-cheeger}]
Let $G = (V,E)$ be an undirected graph and $k\in[n]$. Note that both $ \phi_{k,dir}(G)$ and $\rho_{k,dir}(G)$ are minimizing $\max\{\phid(S_i, T_i): i\in[k]\}$. Since $\phi_{k,dir}(G)$ is minimizing over a larger set, we must have $\phi_{k,dir}(G)\leq \rho_{k,dir}(G)$. 
\par It remains to show the $\rho_{k,dir}(G)\leq 3\phi_{k,dir}(G).$ Let $S_1,\ldots, S_k$ and $T_1,\ldots, T_k$ be non-empty sets of vertices such that $S_i\cap S_j =T_i\cap T_j=\emptyset$ for all $i\neq j\in[k]$ and $\phi_{k,dir}(G) = \max\{\phid(S_i, T_i):i\in [k]\}$.
\par We define $S'_1, \ldots S'_k$ and $T'_1,\ldots T'_k$ as follows.  
If $(\vol(S_i)+\vol(T_i))/3 \leq \vol(S_i\cap T_i)$ , then define $S'_i = T'_i = S_i\cap T_i$. If $\vol(S_i\cap T_i)<(\vol(S_i)+\vol(T_I))/3$, then define $S'_i = S_i\setminus T_i$ and $T'_i = T_i\setminus S_i$. Observe, that $S'_1,\ldots S'_k$ and $T'_1,\ldots, T'_k$ remain disjoint respectively  and satisfy that for all $i\in[k]$ we have $S'_i = T'_i$ or $S'_i\cap T'_i = \emptyset$. Note Lemma~\ref{lem:relating_phi_phisl_phibp}, shows that $\phid(S'_i, T'_i) \leq \min\{\bdir(S_i\setminus T_i, T_i\setminus S_i), \phi(S_i\cap T_i)\}\leq 3\phid(S_i, T_i)$. This implies that $\phid(S'_i, T'_i)\leq 3\phid(S_i, T_i)$, and so $$ \rho_{k, dir}(G)\leq \max\{\phid(S_i, T_i): i \in [k]\}\leq 3\max\{\phid(S'_i, T'_i): i \in [k]\} = \phi_{k,dir}(G).$$
\end{proof}

\section{Relating spectral and vertex expansion}

In this section we study the relation between vertex and spectral expansion for directed and undirected graphs. First, we recall the definition of vertex expansion.

\begin{definition}
   Let $G = (V,E)$ be a $d$-regular digraph on $n$ vertices. For $k\in\N, \delta >0 $ the graph $G$ is a \textbf{$(k, 1+\delta)$-vertex expander} if for every subset $S\subseteq V$ with $\Abs{S}\leq k$, we have $\Abs{N^+(S)}\geq (1+\delta)|S|$.
\end{definition}

It is well known that bounds on $\sigma_2$ imply vertex expansion. 

\begin{theorem}[\cite{Tanner84}, cf. \cite{pseudorandomness_Vadhan}]\label{thm:vadhan_spectral_vertex}
    If $G = (V,E)$ is a $d$-regular digraph and $\sigma_2(G)$ is the second-largest singular value of the normalized adjacency matrix of $G$, then $G$ is a $(n/2, 1+\delta)$ vertex expander, for $\delta = 1-\sigma_2$
\end{theorem}

Alon proved a partial converse for bounded-degree undirected graphs called magnifiers, which is a weaker notion of vertex expansion.

\begin{definition}
    Let $G = (V,E)$ be a $d$-regular undirected graph on $n$ vertices. For $k\in\N, \delta >0 $ the graph $G$ is a $(k,\delta)$-\textbf{magnifier} if for each $S\subseteq V$ with $|S|\leq k$, we have $|N^+(S)\setminus S| \geq \delta |S|.$

\end{definition}

Observe that if $G$ is a $(k, 1+\delta)$-expander, then $G$ is a $(k, \delta)$-magnifier. Alon's result shows:

\begin{theorem}[\cite{AlonEigenvaluesExpanders}]\label{thm:alon_vertex_spectral}
    Let $G = (V, E)$ be a $d$-regular undirected graph and $\mu_2$ be the second-largest eigenvalue of the normalized adjacency matrix of $G$. If $G$ is an $(n/2,\delta)$-magnifier, then $1-\mu_2\geq \Omega(\delta^2/d)$.
\end{theorem}

Observe, that the bound $1-\mu_2\geq \Omega(\delta^2/d)$ is tight by considering an undirected cycle of length $n = 4/\delta$ with $d-2$ self loops at each vertex. It was observed in \cite{pseudorandomness_Vadhan} that by considering the square of the graph, it follows that $1-\sigma_2 = \Omega(\delta^2/d^2)$ yielding a full converse to Theorem~\ref{thm:alon_vertex_spectral} when $d = O(1)$. This argument extends to directed graphs by considering $A^TA$. 

\begin{corollary}[\cite{pseudorandomness_Vadhan}]\label{thm:digraph_vertex_spectral_new}
    Let $G$ be a $d$-regular directed graph and $\sigma_2$ be the second-largest singular value of the normalized adjacency matrix of $G$. If $G$ is a $(n/2, 1+\delta)$-vertex expander, then $1-\sigma_2\geq \Omega\left(\delta^2/d^2\right)$. 
\end{corollary}

\begin{proof}
     Let $G$ be a $d$-regular directed graph and $\sigma_2$ be the second-largest singular value of the normalized adjacency matrix of $G$. Assume $G$ is a $(n/2, 1+\delta)$-vertex expander. We show the graph $G'$ with normalized adjacency matrix $AA^T$ is an $d^2$-regular undirected $(n/2, 1+\delta)$-vertex expander. For every $S\subseteq V(G')$, $N_{G'}(S) = N^{-}_G(N^+_G(S))$. By regularity, $|N^{-}_G(N^+_G(S))| \geq  |N^+_G(S)|$. This implies for all $|S|\leq n/2$, $|N_{G'}(S)|\geq (1+\delta)|S|$. By Theorem~\ref{thm:alon_vertex_spectral}, $1-\sigma_2 = \Omega( \delta^2/d^2)$.
\end{proof}

However this bound on $1-\sigma_2$ has an extra factor of $d$ compared to Alon's bound on $1-\mu_2$ in Theorem~\ref{thm:alon_vertex_spectral}. We show that we can save this factor of $d$ by considering $sl(G)$. 

\begin{theorem}\label{thm:sv_vs_vertex}
      Let $G$ be a $d$-regular directed graph and $\sigma_2$ be the second-largest singular value of the normalized adjacency matrix of $G$. If $G$ is a $(n/2, 1+\delta)$-vertex expander, then $1-\sigma_2\geq \Omega\left(\delta^2/d\right)$.
\end{theorem}

Observe that, like Theorem~\ref{thm:alon_vertex_spectral}, the bound on $1-\sigma_2$ is tight up to a constant factor. Furthermore, we cannot replace $(n/2, 1+\delta)$-vertex expansion with $(n/2, \delta)$-magnification. The complete bipartite graph $K_{n/2,n/2}$ is a $(n/2, 1)$-magnifier, but satisfies $1-\sigma_2 = 0$. 

To prove Theorem~\ref{thm:sv_vs_vertex} we show that if $G$ is a $(n,1+\delta)$-vertex expander, then $sl(G)$ is a $(2n,\Omega(\delta))$-magnifier. 

\begin{lemma}\label{lem:mag_vs_expansion}
    Let $G$ be a directed $d$-regular that is a $(n/2, 1+\delta)$-vertex expander. Then $sl(G)$ is a $(n, \delta/8)$-magnifier.
\end{lemma}

First we show how to deduce Theorem~\ref{thm:sv_vs_vertex} from Lemma~\ref{lem:mag_vs_expansion}

\begin{proof}[Proof of Theorem~\ref{thm:sv_vs_vertex}]
 Let $G$ be a directed $d$-regular and $(n/2, 1+\delta)$-vertex expander. By Lemma~\ref{lem:mag_vs_expansion}, $sl(G)$ is a $(n, \delta/8)$-magnifier. By Theorem~\ref{thm:alon_vertex_spectral},
 $1-\sigma_2(G) = 1-\mu_2(sl(G)) \geq \Omega(\delta^2/d)$.
    
\end{proof}

Now we prove Lemma~\ref{lem:mag_vs_expansion}. 

\begin{proof}[Proof of Lemma~\ref{lem:mag_vs_expansion}]
    We show $sl(G)$ is a $(n, \delta/8)$-magnifier. Let $S\subseteq V(sl(G))$ satisfy that $|S|\leq n$. We must show that $|N(S)\setminus S|\geq \delta|S|/8$. Define $S_L = S\cap L$ and $S_R = S\cap R$. 

      Suppose $|S_L|\geq |S_R| + \delta|S|/8$. Then, 
      \[
      |N(S)\setminus S||\geq |N(S_L)\setminus S_R|= |N(S_L)\setminus S_R|\geq |S_L|-|S_R|\geq \frac{\delta}{8}|S|.
      \]
      By an analogous argument if $|S_R|\geq |S_L| + \delta|S|/8$, then $|N(S)\setminus S|\geq \delta|S|/8$. Hence, we may assume $\bigabs{|S_L|-|S_R|}\leq \delta |S|/8$, which is equivalent to 
      \begin{align}\label{eqn:1}
      |S_R|-\frac{\delta}{8}|S|\leq |S_L||\leq |S_R|+\frac{\delta}{8}|S|.
      \end{align}
      
      Observe,
      \begin{align*}
          |N(S)\setminus S|&\geq |N(S_L)|-|S_R| + |N(S_R)|-|S_L|\\
          &= |N(S_L)|-|S_L|\\
          &\geq (1+\delta)\min\{|S_L|, n/2\} - |S_L|.
      \end{align*}
      If $|S_L|\geq n/2$, then \ref{eqn:1} implies  $|S_L|\leq n(1+\delta/4)/2$. This along with the inequality above implies $|N(S)\setminus S|\geq 3\delta |S|/8$. If $|S_L|\leq n/2$, then \ref{eqn:1} implies $|S|(1-\delta/4)/2\leq |S_L|$. This along with the inequality above and $\delta\leq 1$ implies $|N(S)\setminus S|\geq 3\delta |S|/8$.
      
\end{proof}

\ifnum\conferenceversion=0

\fi

\phantomsection
\addcontentsline{toc}{section}{Bibliography}
{\footnotesize
\bibliographystyle{amsalpha} 
\bibliography{refs}
}
\clearpage


\end{document}